\documentclass[10pt]{article}
\usepackage{mathpazo}
\usepackage{amsmath}
\usepackage{amsthm}
\usepackage{amssymb}
\usepackage{url}
\usepackage{latexsym}
\usepackage[xcolor=pst]{pstricks}
\usepackage{graphicx, pst-plot, pst-node, pst-text, pst-tree}
\usepackage{titlefoot}
\usepackage[small]{titlesec}
\usepackage{units} % for \nicefrac
\usepackage[small,it]{caption}

\usepackage{color}
\definecolor{lightgray}{rgb}{0.8, 0.8, 0.8}
\definecolor{darkgray}{rgb}{0.7, 0.7, 0.7}
\definecolor{darkblue}{rgb}{0, 0, .4}

\usepackage[bookmarks]{hyperref}
\hypersetup{
        colorlinks=true,
        linkcolor=darkblue,
        anchorcolor=darkblue,
        citecolor=darkblue,
        urlcolor=darkblue,
        pdfpagemode=UseThumbs,
        pdftitle={Large Infinite Antichains of Permutations},
        pdfsubject={Combinatorics},
        pdfauthor={Albert, Brignall, and Vatter},
}

\newcounter{todocounter}

\newcommand{\minisec}[1]{\noindent{\sc #1.}}

\theoremstyle{plain}
\newtheorem{theorem}{Theorem}[section]
\newtheorem{proposition}[theorem]{Proposition}

\theoremstyle{definition}

\makeatletter
\newfont{\footsc}{cmcsc10 at 8truept}
\newfont{\footbf}{cmbx10 at 8truept}
\newfont{\footrm}{cmr10 at 10truept}
\renewenvironment{abstract}%
                {
                  \begin{list}{}%
                     {\setlength{\rightmargin}{1in}%
                      \setlength{\leftmargin}{1in}}%
                   \item[]\ignorespaces\begin{small}}%
                 {\end{small}\unskip\end{list}}

\newcommand{\Av}{\operatorname{Av}}

\newcommand{\C}{\mathcal{C}}

\newcommand{\gr}{\mathrm{gr}}
\newcommand{\lgr}{\underline{\gr}}
\newcommand{\ugr}{\overline{\gr}}

\newcommand{\st}{\::\:}

\newcommand{\lesup}{\textrm{\tiny $\le$}}
\datefoot{\today}
\amssubj{05A05, 05A15}
\newpagestyle{main}[\small]{
        \headrule
        \sethead[\usepage][][]
        {\sc Large Infinite Antichains of Permutations}{}{\usepage}}

\title{\sc Large Infinite Antichains of Permutations}
\author{%
Michael H. Albert\\[-0.25ex]
\small Department of Computer Science\\[-0.5ex]
\small University of Otago\\[-0.5ex]
\small Dunedin, New Zealand\\[1.5ex]
Robert Brignall\footnotemark[\value{footnote}]\footnote{Brignall was partially supported by the EPSRC Grant EP/J006130/1.}\\[-0.25ex]
\small Department of Mathematics and Statistics\\[-0.5ex]
\small The Open University\\[-0.5ex]
\small Milton Keynes, England\\[1.5ex]
Vincent Vatter\footnotemark[\value{footnote}]\footnote{Vatter was partially supported by the NSA Young Investigator Grant H98230-12-1-0207.}\\[-0.25ex]
\small Department of Mathematics\\[-0.5ex]
\small University of Florida\\[-0.5ex]
\small Gainesville, Florida USA\\[-1.5ex]
}

\titleformat{\section}
        {\large\sc}
        {\thesection.}{1em}{}

\date{}

\begin{document}
\maketitle

\pagestyle{main}

\begin{abstract}
Infinite antichains of permutations have long been used to construct interesting permutation classes and counterexamples.  We prove the existence and detail the construction of infinite antichains with arbitrarily large growth rates.  As a consequence, we show that every proper permutation class is contained in a class with a rational generating function. While this result implies the conclusion of the Marcus-Tardos theorem, that theorem is used in our proof.
\end{abstract}

\section{Introduction}

%We make no larger claims than that this construction clears up a trivial bit of minutiae; however, we do feel that the general construction techniques, which are novel, may well serve to clear up unforeseen future bits of minutiae.

The permutation $\pi$ of length $n$ {\it contains\/} the permutation $\sigma$ of length $k$, written $\sigma\le\pi$, if $\pi$ has a not-necessarily-contiguous subsequence of length $k$ in the same relative order as $\sigma$.  For example, $\pi=391867452$ (written in list, or one-line notation) contains $\sigma=51342$, as can be seen by considering the subsequence $91672$ ($=\pi(2),\pi(3),\pi(5),\pi(6),\pi(9)$).

Since the earliest studies on permutation patterns, it has been known that the set of permutations ordered by containment contains infinite antichains (sets of pairwise incomparable elements): Pratt~\cite{pratt:computing-permu:} constructed such a set in his studies of double-ended queues in 1973.  This fact should not be surprising.  While the celebrated Minor Theorem of Robertson and Seymour~\cite{robertson:graph-minors-i-xx:} shows that graphs \emph{ordered by the minor relation} do not contain an infinite antichain, the minor relation is not analogous to the containment order on permutations.  The containment order on permutations is much more similar to the induced subgraph order on graphs, and under this order graphs clearly do contain infinite antichains; one such antichain consists of all cycles, while another consists of the \emph{split-end paths} shown in Figure~\ref{fig-U-graphs}.
% In fact, these are the only two fundamental antichains, right?

\begin{figure}
\begin{center}
\psset{xunit=0.01in, yunit=0.01in}
\psset{linewidth=0.005in}
\begin{pspicture}(40,0)(320,40)
\pscircle*(40,0){0.04in}
\pscircle*(40,40){0.04in}
\pscircle*(70,20){0.04in}
\pscircle*(100,20){0.04in}
\pscircle*(130,20){0.04in}
\pscircle*(160,20){0.04in}
\rput[c](181,20){$\cdots$}
\pscircle*(200,20){0.04in}
\pscircle*(230,20){0.04in}
\pscircle*(260,20){0.04in}
\pscircle*(290,20){0.04in}
\pscircle*(320,0){0.04in}
\pscircle*(320,40){0.04in}
\psline(40,0)(70,20)
\psline(40,40)(70,20)
\psline(70,20)(160,20)
\psline(200,20)(290,20)
\psline(290,20)(320,0)
\psline(290,20)(320,40)
\end{pspicture}
\end{center}
\caption{The antichain of split-end paths.}\label{fig-U-graphs}
\end{figure}
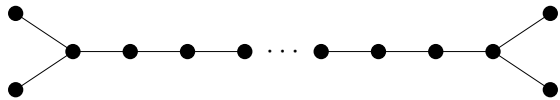

In fact, from the graph antichain of split-end paths we can easily construct an infinite antichain of permutations.  Given a permutation $\pi$ of length $n$, its \emph{inversion graph} is the graph $G_\pi$ on the vertices $\{1,\dots,n\}$ where $i\sim j$ if and only if $i<j$ and $\pi(i)>\pi(j)$.  If $\sigma\le\pi$ then $G_\sigma$ is an induced subgraph of $G_\pi$ (although the reverse does not necessarily hold), so to construct an infinite antichain of permutations, we need only find a set of permutations whose inversion graphs are split-end paths.

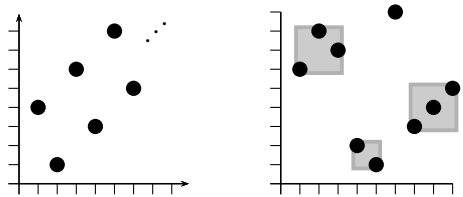
\begin{figure}
\begin{center}
\begin{tabular}{ccc}
\psset{xunit=0.01in, yunit=0.01in}
\psset{linewidth=0.005in}
\begin{pspicture}(0,0)(90,90)
\psaxes[dy=10, Dy=1, dx=10, Dx=1,tickstyle=bottom,showorigin=false,labels=none]{->}(0,0)(89,89)
\pscircle*(10,40){4.0\psxunit}
\pscircle*(20,10){4.0\psxunit}
\pscircle*(30,60){4.0\psxunit}
\pscircle*(40,30){4.0\psxunit}
\pscircle*(50,80){4.0\psxunit}
\pscircle*(60,50){4.0\psxunit}
%\pscircle*(70,100){4.0\psxunit}
%\pscircle*(80,70){4.0\psxunit}
%\pstextpath[c]{\psline[linecolor=white](85,98)(95,108)}{$\dots$}
\pstextpath[c]{\psline[linecolor=white](65,78)(75,88)}{$\dots$}
\end{pspicture}
&\rule{10pt}{0pt}&
\psset{xunit=0.01in, yunit=0.01in}
\psset{linewidth=0.005in}
\begin{pspicture}(0,0)(90,90)
\psaxes[dy=10,Dy=1,dx=10,Dx=1,tickstyle=bottom,showorigin=false,labels=none](0,0)(90,90)
\psframe[linecolor=darkgray,fillstyle=solid,fillcolor=lightgray,linewidth=0.02in](7,57)(33,83)
\psframe[linecolor=darkgray,fillstyle=solid,fillcolor=lightgray,linewidth=0.02in](37,7)(53,23)
\psframe[linecolor=darkgray,fillstyle=solid,fillcolor=lightgray,linewidth=0.02in](67,27)(93,53)
\pscircle*(10,60){0.04in}
\pscircle*(20,80){0.04in}
\pscircle*(30,70){0.04in}
\pscircle*(40,20){0.04in}
\pscircle*(50,10){0.04in}
\pscircle*(60,90){0.04in}
\pscircle*(70,30){0.04in}
\pscircle*(80,40){0.04in}
\pscircle*(90,50){0.04in}
\end{pspicture}
\end{tabular}
\end{center}
\caption{On the left, the plot of the increasing oscillating sequence.  On the right, the plot of the permutation $687219345=3142[132,21,1,123]$.  }\label{fig-int-and-osc}
\end{figure}

To this end, we define the {\it increasing oscillating sequence\/} as the sequence
$$
4,1,6,3,8,5,\dots,2k+2,2k-1,\dots
$$
(see Figure~\ref{fig-int-and-osc}).  For $m\ge 4$, let $\sigma_m$ denote the permutation in the same relative order as
\begin{itemize}
\item the first $m$ entries of the increasing oscillating sequence if $m$ is even, or
\item the least $m$ entries (by value) of the increasing oscillating sequence if $m$ is odd.
\end{itemize}
The inversion graph of $\sigma_m$ is a path of length $m$, so to find a set of permutations whose inversion graphs are split-end paths, we merely need to ``blow up'' the ``endpoints'' of $\sigma_m$.  As this operation is crucial to our constructions, we describe it in some detail.

An \emph{interval} in the permutation $\pi$ is a set of contiguous indices $I=\{a,a+1,\dots,b\}$ such that the set $\{\pi(i)\st i\in I\}$ is also contiguous.  Every permutation $\pi$ of length $n$ has \emph{trivial intervals} of lengths $0$, $1$, and $n$, and other intervals are called \emph{proper}.  A permutation with no proper intervals is called \emph{simple}.

Simple permutations are precisely those that do not arise from a non-trivial inflation, in the following sense.  Given a permutation $\sigma$ of length $m$ and nonempty permutations $\alpha_1,\dots,\alpha_m$, the \emph{inflation} of $\sigma$ by $\alpha_1,\dots,\alpha_m$,  denoted $\sigma[\alpha_1,\dots,\alpha_m]$, is the permutation of length $|\alpha_1|+\cdots+|\alpha_m|$ obtained by replacing each entry $\sigma(i)$ by an interval that is order isomorphic to $\alpha_i$ in such a way that the intervals are order isomorphic to $\sigma$.  For example,
\[
2413[1,132,321,12]=4\ 798\ 321\ 56.
\]
It can be established that every permutation is the inflation of a unique simple permutation, called its \emph{simple quotient} and, moreover, that the intervals in such an inflation are unique unless the simple quotient is $12$ or $21$.

\begin{figure}
\begin{center}
\begin{tabular}{ccccccc}
\psset{xunit=0.01in, yunit=0.01in}
\psset{linewidth=0.005in}
\begin{pspicture}(0,0)(110,110)
\psaxes[dy=10,Dy=1,dx=10,Dx=1,tickstyle=bottom,showorigin=false,labels=none](0,0)(110,110)
\psframe[linecolor=darkgray,fillstyle=solid,fillcolor=lightgray,linewidth=0.02in](17,7)(33,23)
\psframe[linecolor=darkgray,fillstyle=solid,fillcolor=lightgray,linewidth=0.02in](97,87)(113,103)
\pscircle*(10,40){0.04in}
\pscircle*(20,10){0.04in}
\pscircle*(30,20){0.04in}
\pscircle*(40,60){0.04in}
\pscircle*(50,30){0.04in}
\pscircle*(60,80){0.04in}
\pscircle*(70,50){0.04in}
\pscircle*(80,110){0.04in}
\pscircle*(90,70){0.04in}
\pscircle*(100,90){0.04in}
\pscircle*(110,100){0.04in}
\end{pspicture}
&
\rule{3pt}{0pt}
&
\psset{xunit=0.01in, yunit=0.01in}
\psset{linewidth=0.005in}
\begin{pspicture}(0,0)(120,120)
\psaxes[dy=10,Dy=1,dx=10,Dx=1,tickstyle=bottom,showorigin=false,labels=none](0,0)(120,120)
\psframe[linecolor=darkgray,fillstyle=solid,fillcolor=lightgray,linewidth=0.02in](17,7)(33,23)
\psframe[linecolor=darkgray,fillstyle=solid,fillcolor=lightgray,linewidth=0.02in](97,107)(113,123)
\pscircle*(10,40){0.04in}
\pscircle*(20,10){0.04in}
\pscircle*(30,20){0.04in}
\pscircle*(40,60){0.04in}
\pscircle*(50,30){0.04in}
\pscircle*(60,80){0.04in}
\pscircle*(70,50){0.04in}
\pscircle*(80,100){0.04in}
\pscircle*(90,70){0.04in}
\pscircle*(100,110){0.04in}
\pscircle*(110,120){0.04in}
\pscircle*(120,90){0.04in}
\end{pspicture}
&
\rule{3pt}{0pt}
&
\psset{xunit=0.01in, yunit=0.01in}
\psset{linewidth=0.005in}
\begin{pspicture}(0,0)(110,110)
\psaxes[dy=10,Dy=1,dx=10,Dx=1,tickstyle=bottom,showorigin=false,labels=none](0,0)(110,110)
\pscircle*(10,40){0.04in}
\pscircle*(20,10){0.04in}
\pscircle*(30,20){0.04in}
\pscircle*(40,60){0.04in}
\pscircle*(50,30){0.04in}
\pscircle*(60,80){0.04in}
\pscircle*(70,50){0.04in}
\pscircle*(80,110){0.04in}
\pscircle*(90,70){0.04in}
\pscircle*(100,90){0.04in}
\pscircle*(110,100){0.04in}
\psline(10,40)(20,10)
\psline(10,40)(30,20)
\psline(10,40)(50,30)
\psline(40,60)(50,30)
\psline(40,60)(70,50)
\psline(60,80)(70,50)
\psline(60,80)(90,70)
\psline(80,110)(90,70)
\psline(80,110)(100,90)
\psline(80,110)(110,100)
\end{pspicture}
&
\rule{3pt}{0pt}
&
\psset{xunit=0.01in, yunit=0.01in}
\psset{linewidth=0.005in}
\begin{pspicture}(0,0)(120,120)
\psaxes[dy=10,Dy=1,dx=10,Dx=1,tickstyle=bottom,showorigin=false,labels=none](0,0)(120,120)
\pscircle*(10,40){0.04in}
\pscircle*(20,10){0.04in}
\pscircle*(30,20){0.04in}
\pscircle*(40,60){0.04in}
\pscircle*(50,30){0.04in}
\pscircle*(60,80){0.04in}
\pscircle*(70,50){0.04in}
\pscircle*(80,100){0.04in}
\pscircle*(90,70){0.04in}
\pscircle*(100,110){0.04in}
\pscircle*(110,120){0.04in}
\pscircle*(120,90){0.04in}
\psline(10,40)(20,10)
\psline(10,40)(30,20)
\psline(10,40)(50,30)
\psline(40,60)(50,30)
\psline(40,60)(70,50)
\psline(60,80)(70,50)
\psline(60,80)(90,70)
\psline(80,100)(90,70)
\psline(80,100)(120,90)
\psline(100,110)(120,90)
\psline(110,120)(120,90)
\end{pspicture}
\end{tabular}
\end{center}
\caption{Two members of the infinite antichain $U$ on the left, with their inversion graphs displayed on the right.}\label{fig-first-antichains}
\end{figure}
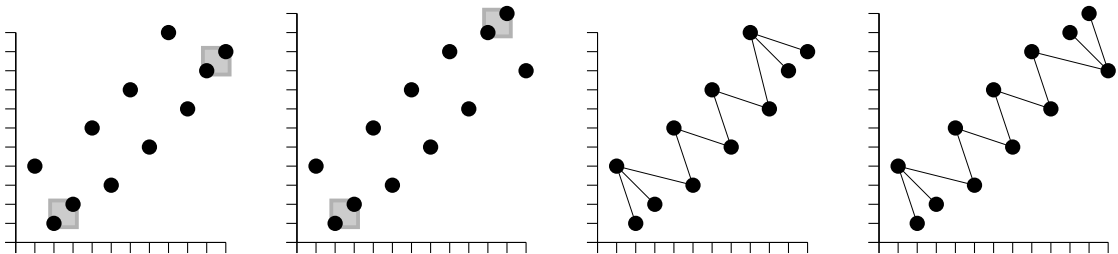

We can now describe an infinite antichain of permutations, which we call $U$, whose inversion graphs are split-end paths.  The elements of $U$ are formed by inflating the ``endpoints'' of $\sigma_m$ by the interval $12$ for all $m\ge 4$: if $m$ is even we inflate the least and greatest entries of $\sigma_m$, while if $m$ is odd we inflate the least and rightmost entries of $\sigma_m$.  This antichain contains one permutation of each length $n\ge6$.

In the remainder of this introduction, we give some background on permutation classes, and in particular, the role that infinite antichains have played in the investigation of permutation classes.  In the following two sections we detail our construction of large infinite antichains and prove that every proper permutation class is contained in a class with a rational generating function.

\minisec{Permutation classes}
A \emph{permutation class} (which we often abbreviate to \emph{class}) is a downset (or, to use a graph-theoretic notion, hereditary property) of permutations under the containment order; thus if $\C$ is a permutation class, $\pi\in\C$, and $\sigma\le\pi$ then $\sigma\in\C$.  Permutation classes can be specified by the minimal permutations \emph{not} in the class, which necessarily form an antichain and which we call the \emph{basis}.  In other words, for each permutation class there is a (possibly infinite) antichain $B$ such that
$$
\C=\Av(B)=\{\pi: \pi \not \geq\beta\mbox{ for all } \beta \in B\}.
$$

Permutation classes can also be specified more positively in terms of the permutations they do contain.  Given a set $X$ of permutations, we define its \emph{closure} as
$$
X^\lesup=\{\sigma\st \sigma\le\pi\mbox{ for some }\pi\in X\}.
$$

Much of the research on permutation classes has focused on their exact and asymptotic enumeration.  Given a set $X$ of permutations (notably a permutation class, or an infinite antichain), we denote by $X_n$ the set of permutations in $X$ of length $n$.  The \emph{generating function} of $X$ is then
$$
\sum_{\mbox{\footnotesize nonempty }\pi\in X} x^{|\pi|}=\sum_{n\ge 1}|X_n|x^n.
$$
(As a matter of convention, except when explicitly stated otherwise, we do not include the empty permutation in our generating functions.)  The \emph{upper and lower growth rates} of this set $X$ are defined, respectively, by
\begin{eqnarray*}
\ugr(X)&=&\limsup_{n\rightarrow\infty} \sqrt[n]{|X_n|},\\
\lgr(X)&=&\liminf_{n\rightarrow\infty} \sqrt[n]{|X_n|}.
\end{eqnarray*}
If $\lgr(X)=\ugr(X)$, then we call this quantity the \emph{proper growth rate} of $X$ and denote it by $\gr(X)$.  Pringsheim's Theorem, below, connects exact and asymptotic enumeration.

\newtheorem*{pringsheimsthm}{Pringsheim's Theorem}
\begin{pringsheimsthm}[see Flajolet and Sedgewick~{\cite[Section IV.3]{flajolet:analytic-combin:}}]
The upper growth rate of the set $X$ of permutations is equal to the reciprocal of the least positive singularity of its generating function.
\end{pringsheimsthm}

For permutation classes, the Marcus-Tardos Theorem~\cite{marcus:excluded-permut:} (formerly the Stanley-Wilf Conjecture) states every proper permutation class has a \emph{finite} upper growth rate; here \emph{proper} means that the class omits at least one permutation.  It is not known whether proper permutation classes have proper growth rates.

\minisec{The role of infinite antichains in the study of permutation classes}
Infinite antichains have been used frequently to construct exotic permutation classes with counterintuitive properties.  Perhaps the first such result is due to Murphy~\cite{murphy:restricted-perm:}.  Noonan and Zeilberger~\cite{noonan:the-enumeration:} had conjectured%
\footnote{Technically, Noonan and Zeilberger made a seemingly stronger conjecture, but Atkinson~\cite{atkinson:restricted-perm:} proved that their stronger claim is equivalent to the form presented here.}
that every finitely based permutation class has a holonomic (or $D$-finite) generating function%
\footnote{It should be noted that Zeilberger has since repudiated his conjecture.  As reported by Elder and Vatter~\cite{elder:problems-and-co:}, at the \emph{Third International Conference on Permutation Patterns}, in 2005, Zeilberger conjectured that there are finitely based permutation classes with non-holonomic generating functions.  He went on to speculate that ``not even God knows $\Av_{1000}(1324)$.''}%
, meaning that the generating function for the class and all its derivatives generate a finite dimensional vector space over $\mathbb{C}(x)$.  Murphy showed that this ``finitely based'' hypothesis is essential, by observing that because there are infinite antichains of permutations, there are uncountably many permutation classes with different enumerations but only countably many holonomic generating functions with integer coefficients.

Another example is the membership problem: given a basis $B$ and a permutation $\pi$ of length $n$, how long does it take (in the worst case, as a function of $n$) to decide if $\pi\in\Av(B)$?  As there are only countably many algorithms, it follows trivially that there are permutation classes with undecidable membership problems.

A third and final example of the usefulness of infinite antichains concerns a conjecture of Balogh, Bollob\'as, and Morris~\cite{balogh:hereditary-prop:ordgraphs}.  They studied \emph{ordered graphs} under the induced subgraph order --- given graphs $G$ and $H$ on $\{1,\dots,n\}$ and $\{1,\dots,k\}$, respectively, we say that $H$ is an {\it ordered subgraph\/} of $G$ if there is an increasing injection $f:\{1,\dots,k\}\rightarrow\{1,\dots,n\}$ such that $i\sim_H j$ if and only if $f(i)\sim_G f(j)$.  Balogh, Bollob\'as, and Morris had conjectured that every upper growth rate of a hereditary property of ordered graphs is algebraic.  Every permutation class can be viewed as a hereditary property of ordered graphs, simply by considering ordered versions of the inversion graphs of the permutations in the class, and so their conjecture was stronger than the statement that every upper growth rate of a permutation class is algebraic.  Albert and Linton~\cite{albert:growing-at-a-pe:} disproved this conjecture, with a construction involving variations on the antichain $U$.  Vatter~\cite{vatter:permutation-cla} refined this technique to show that every real number greater than $2.48188$ is the upper growth rate of a permutation class.  Klazar~\cite{klazar:overview-of-som}, among others, has suggested that the Balogh-Bollob\'as-Morris Conjecture may still hold for finitely based classes of permutations / hereditary properties of ordered graphs, but, as with the Noonan-Zeilberger Conjecture, it is now clear that the finite basis hypothesis is necessary.

\section{The Construction}\label{sec-construction}

At the \emph{Third International Conference on Permutation Patterns} (see Elder and Vatter~\cite{elder:problems-and-co:}), the first author of this article asked whether there are antichains with arbitrarily large upper growth rates.  We begin by constructing such antichains.

Let $A$ be any antichain of permutations (we will typically take $A$ to be finite, but this does not matter for the construction), and $\alpha$ any permutation which is not contained in any member of $A$, i.e., $\alpha\notin A^\lesup$.  We construct the antichain $U_{A,\alpha}$ by inflating the oscillations $\sigma_m$ for $m\ge 4$ much as we did to form $U$:
\begin{itemize}
\item if $m$ is even, inflate the least and greatest entries of $\sigma_m$ by $\alpha$, and inflate all other entries of $\sigma_m$ by arbitrary elements of $A$;
\item if $m$ is odd, inflate the least and rightmost entries of $\sigma_m$ by $\alpha$, and inflate all other entries of $\sigma_m$ by arbitrary elements of $A$.
\end{itemize}
In this notation, the antichain $U$ we constructed in the Introduction is $U_{\{1\},12}$.

\begin{proposition}
\label{prop-antichain}
For any antichain $A$ and permutation $\alpha\notin A^\lesup$, the set $U_{A,\alpha}$ forms an infinite antichain.
\end{proposition}
\begin{proof}
Suppose that $\sigma\le\pi$ for $\sigma,\pi\in U_{A,\alpha}$.  By the definition of simple permutations, it follows that $\sigma$ must embed into a single interval of $\pi$, or the simple quotient of $\sigma$ (the increasing oscillation that was inflated to form $\sigma$) must embed into the simple quotient of $\pi$.  The first possibility cannot occur because $\alpha\le\sigma$, and thus $\sigma$ could only embed into the two intervals of $\pi$ containing $\alpha$, but this is impossible because $\sigma>\alpha$.

Now consider the second possibility.  Both $\sigma$ and $\pi$ were formed by inflating the smallest and  the greatest/rightmost entries of their simple quotients by $\alpha$. Since $\alpha$ cannot embed in any other interval, the smallest maximal proper interval of $\sigma$ must embed into the smallest maximal proper interval of $\pi$, and similarly with the greatest/rightmost maximal proper intervals.  However, since the simple quotients of $\sigma$ and $\pi$ are increasing oscillations, this can only be achieved if they are of the same length. Finally, since the remaining maximal proper intervals of $\sigma$ and $\pi$ are order isomorphic to permutations in $A$, and since $A$ is an antichain, we conclude that the only way $\sigma$ could embed into $\pi$ is if $\sigma=\pi$.
\end{proof}

Now consider the enumeration of $U_{A,\alpha}$.  The generating function for the increasing oscillations of length at least $4$ which are inflated in the construction of $U_{A,\alpha}$ is
$$
\frac{x^4}{1-x}=x^4+x^5+\cdots,
$$
so if $a(x)$ denotes the generating function for $A$ and $\alpha$ has length $k+1$, the generating function for $U_{A,\alpha}$ is
$$
\frac{x^{2k+2}a^2(x)}{1-a(x)}.
$$
By Pringsheim's Theorem, the upper growth rate of $U_{A,\alpha}$ is the reciprocal of the least positive solution to $a(x)=1$ or possibly, in the case where $A$ is infinite, the reciprocal of the least singularity of $a(x)$.

Now set $A=S_k$, the set of all permutation of length $k$ (which is trivially an antichain), and take $\alpha$ to be any permutation of length $k+1$.  We see that $a(x)=k!x^k$, so the least positive solution of $a(x)=1$ is
$$
x=\frac{1}{\sqrt[k]{k!}}\sim \frac{e}{k}
$$
by Stirling's Formula.  We therefore conclude that $\ugr(U_{S_k,\alpha})\rightarrow\infty$ as $k\rightarrow\infty$.

We have constructed antichains of arbitrarily large \emph{upper} growth rates, but every element of $U_{S_k,\alpha}$ has length congruent to $2$ modulo $k$, so these antichains do not have proper growth rates.  To fix this problem, we alter our construction slightly.  Choose an arbitrary permutation $\tau$ of length $k-1$, and set
$$
A_\tau=\{\tau\}\cup\{\pi\in S_k\st \tau\not\le \pi\}.
$$
It is not difficult to see that every permutation of length $k-1$ is contained in $(k-1)^2+1=k^2-2k+2$ permutations of length $k$, so the generating function of $A_\tau$ is
$$
x^{k-1}+\left(k!-k^2+2k-2\right)x^k.
$$
Clearly we still have $\ugr(U_{A_\tau, \alpha})\rightarrow\infty$ as $k\rightarrow\infty$, but now we claim that this antichain has a proper growth rate.

\begin{proposition}
\label{prop-antichain-gr}
For any permutations $\tau$ of length $k-1$ and $\alpha$ of length $k+1$, the antichain $U_{A_\tau,\alpha}$ (with $A_\tau$ constructed as above) has a proper growth rate, which tends to $\infty$ as $k\rightarrow\infty$.
\end{proposition}
\begin{proof}
Let $c=\left(k!-k^2+2k-2\right)$, so the generating function for $U_{A_\tau,\alpha}$ is given by
$$
\frac{x^{2k}\left(x^{k-1}+cx^k\right)^2}{1-x^{k-1}-cx^k}.
$$
By Pringsheim's Theorem, we know that the upper growth rate of the corresponding sequence is determined by the least positive root of the denominator of this generating function, say $r$.  We will be done if we can establish that $r$ is the unique root of modulus $r$.  Suppose to the contrary that this denominator has another root of modulus $r$, say $\omega r$.  Then we see that
$$
1-r^{k-1}-cr^k=1-\omega^{k-1}r^{k-1}-c\omega^k r^k=0.
$$
Therefore, by equating both sides and canceling $r^{k-1}$, we see that
$$
\omega^{k-1}+c\omega^{k}r=1+cr.
$$
Now compare the moduli of both sides of this equation.  Since both terms on the right-hand side are positive and real, the only way the moduli could be equal is if $\omega^{k-1}=\omega^k=1$, which implies that $\omega=1$, as desired.
\end{proof}

\section{Rational Superclasses}\label{sec-rat-superclasses}

Suppose that we are given a proper permutation class $\C$.  Choose an integer $k$ and permutations $\alpha$ and $\tau$ so that three conditions are satisfied:
\begin{enumerate}
\item[(A1)] $\alpha$ has length $k+1$ and does not lie in $\C$,
\item[(A2)] $\tau$ has length $k-1$ and is not contained in $\alpha$, and
\item[(A3)] $\gr(U_{A_\tau,\alpha})>\ugr(\C)$.
\end{enumerate}
Note that (A1) can be satisfied because $\C$ is proper, (A2) can be satisfied because $\alpha$ contains at most ${k+1\choose 2}$ permutations of length $k-1$, which is less than $(k-1)!$ for $k\geq 6$, and (A3) can be ensured by Proposition~\ref{prop-antichain-gr} and the Marcus-Tardos Theorem.

Because $\alpha$ is not a member of $\C$ but is contained in every member of the antichain $U_{A_\tau,\alpha}$, we see that $U_{A_\tau, \alpha}\cap \C=\emptyset$.

We can now outline our approach.  Define
$$
\C_{A_\tau, \alpha}=\C\cup U^\lesup_{A_\tau,\alpha}.
$$
By our previous observations, $U_{A_\tau,\alpha}$ is a \emph{maximal} antichain in $\C_{A_\tau,\alpha}$, meaning that no element of $U_{A_\tau,\alpha}$ is contained in any other element of $\C_{A_\tau,\alpha}$.  It follows that $\C_{A_\tau,\alpha}\setminus X$ forms a permutation class for every subset $X\subseteq U_{A_\tau,\alpha}$.  By (A3) there is some integer $N$ so that for every $n\ge N$, $U_{A_\tau,\alpha}$ contains more permutations of length $n$ than $\C$.  For each $n\ge N$, let $X_n$ consist of $|\C_n|$ permutations of length $n$ chosen from $U_{A_\tau,\alpha}$.  Now form the class
$$
\C_{\textrm{rat}}=\C_{A_\tau,\alpha}\setminus (X_N\cup X_{N+1}\cup\cdots).
$$
This class contains precisely as many permutation of each length $n\ge N$ as $U^\lesup_{A_\tau,\alpha}$.  Therefore the generating function of $\C_{\textrm{rat}}$ will differ from that of $U_{A_\tau,\alpha}$ by a polynomial of degree at most $N-1$, and thus will be rational if the generating function of $U_{A_\tau,\alpha}$ is rational.  Thus the following proposition will complete the proof of our desired result.

\begin{proposition}
\label{prop-antichain-rat}
Suppose that the permutation $\alpha$ of length $k+1$ does not contain the permutation $\tau$ of length $k-1$ and that $A_\tau$ is constructed as described in the previous section.  Then the generating function of $U^\lesup_{A_\tau,\alpha}$ is rational.
\end{proposition}
\begin{proof}
We first briefly review the \emph{sum decomposition} of permutations.  Given two permutations $\pi$ and $\sigma$, we denote the inflation $12[\pi,\sigma]$ by $\pi\oplus\sigma$.  A permutation is said to be {\it sum indecomposable\/} if it cannot be written as the direct sum of two shorter permutations%
\footnote{Equivalently, the permutation $\pi$ is sum indecomposable if and only if its inversion graph $G_\pi$ is connected, and thus these permutations are themselves called \emph{connected} by some authors.}.
Note that every permutation has a unique representation as a sum of sum indecomposable permutations.

It is easy to verify that deleting any entry from an increasing oscillation results in either a shorter increasing oscillation or a sum decomposable permutation.  Indeed, the only entries one may delete to create a shorter increasing oscillation are those which we inflate by $\alpha$ in order to form the antichain $U_{A_\tau,\alpha}$.  We need one more observation: because $\tau\not\le \alpha$ by (A2), if we delete any entries from $\alpha$, we obtain a permutation in $A_\tau^\lesup$.

Let $a$ denote the generating function of $A^\lesup_\tau$.  Clearly $a$ is a polynomial; to be more precise, it is given by
$$
a = x+2x^2+6x^3+\cdots+(k-1)!x^{k-1}+\left(k!-k^2+2k-2\right)x^k,
$$
but this level of detail is unnecessary for our analysis.  Now divide the elements of $U^\lesup_{A_\tau,\alpha}$ into two groups:
\begin{itemize}
\item[(C1)] sum indecomposable permutations which contain two occurrences of $\alpha$, and
\item[(C2)] all other permutations.
\end{itemize}

We begin with (C1).  By definition, all such permutations have two occurrences of $\alpha$, which the generating function accounts for with an $x^{2k+2}$ factor.  In addition to these two intervals, in order for these permutations to be sum indecomposable, all other entries of the underlying increasing oscillating sequence must be inflated by nonempty permutations from $A^\lesup_\tau$.  Thus the generating function for permutations in (C1) is
$$
\frac{x^{2k+2}a^2}{1-a}.
$$
Note that this generating function includes the contribution of the antichain itself.

Next we consider permutations of the form (C2).  These permutations can be decomposed in the form $\beta\oplus\mu\oplus\eta$, where:
\begin{itemize}
\item $\beta$ is either empty or is formed by inflating the least entry of an increasing oscillation by $\alpha$ and all other entries by nonempty permutations in $A^\lesup_\tau$,
\item $\mu$ is either empty or the sum of increasing oscillations with all entries inflated by $A^\lesup_\tau$, and
\item $\eta$ is either empty or is formed by inflating the greatest or rightmost entry of an increasing oscillation by $\alpha$ and all other entries by nonempty permutations in $A^\lesup_\tau$.
\end{itemize}
Note that both $\beta$ and $\eta$ are sum indecomposable.  We determine generating functions for each of these three parts separately.

For each $m\ge 1$, there is precisely one increasing oscillation of length $m$ which may be inflated to obtain a suitable ``beginning'' $\beta$.  Therefore the generating function for these permutations is
$$
1+\frac{x^{k+1}}{1-a},
$$
where the $1$ accounts for the fact that $\beta$ may be empty.  To count suitable ``middles'' $\mu$, we first count the sum indecomposable permutations of this form.  The sequence of sum indecomposable permutations of length $m$ contained in some increasing oscillation is easily seen to be enumerated by the sequence $1,1,2,2,\dots$ for $m\ge 1$, and thus has the generating function $(x+x^3)/(1-x)$.  Thus the generating function for inflations of these permutations by intervals from $A^\lesup_\tau$ is can be obtained by substituting $a$ for $x$ in this generating function, and the generating function for (possibly empty) sums of these permutations is
$$
\frac{1}{1-(a+a^3)/(1-a)}
=
\frac{1-a}{1-2a-a^3}.
$$
Finally, we need to count ``endings'' $\eta$.  Inflations of the unique increasing oscillation of length $1$ contribute $x^{k+1}$ to this generating function.  Inflations of the unique increasing oscillation of length $2$ (the permutation $21$) contribute $2x^{k+1}a$ to this generating function, as either entry may be inflated by $\alpha$.  Finally, there are $2$ increasing oscillations of each length $m\ge 3$, and each of these increasing oscillations has a unique entry which may inflated by $\alpha$, so these contribute
$$
2\frac{x^{k+1}a^2}{1-a}
$$
to the generating function.  Combining these terms and accounting for the possibility that $\eta$ is empty shows that the generating function for suitable endings $\eta$ is
$$
1+x^{k+1}+2\frac{x^{k+1}a}{1-a}.
$$

Now we can compute that the generating function of $U^\lesup_{A_\tau,\alpha}$ is
$$
\underbrace{\frac{x^{2k+2}a^2}{1-a}}_{\mbox{\footnotesize type (C1)}}
+
\underbrace{\left[
\left(1+\frac{x^{k+1}}{1-a}\right)
\left(\frac{1-a}{1-2a-a^3}\right)
\left(1+x^{k+1}+2\frac{x^{k+1}a}{1-a}\right)
-1
\right]}_{\footnotesize \mbox{type (C2), with the empty permutation removed}},
$$
which is clearly rational, as desired.
\end{proof}

Our main result now follows from our previous remarks and Proposition~\ref{prop-antichain-rat}.

\begin{theorem}
\label{thm-rat-superclasses}
Every proper permutation class is contained in a permutation class with a rational generating function.
\end{theorem}

Note that Theorem~\ref{thm-rat-superclasses} is prima facie stronger than the Marcus-Tardos Theorem.  Of course, the Marcus-Tardos Theorem has played a crucial role in its proof.

Lest the reader be concerned that one could do even better than the construction in Section~\ref{sec-construction} and construct super-exponential antichains, we conclude with the following corollary of the Marcus-Tardos Theorem.

\begin{proposition}
Every antichain of permutations has a finite upper growth rate.
\end{proposition}
%\begin{proof}
%Let $A$ be an antichain of permutations and choose some $\alpha\in A$.  Because $A$ is an antichain, $\alpha\notin(A\setminus\{\alpha\})^\lesup$.  Therefore $(A\setminus\{\alpha\})^\lesup$ has a finite upper growth rate, which serves as an upper bound for the upper growth rate of $A$ itself.
%\end{proof}
%
\begin{proof}
Let $A$ be an antichain of permutations and choose some $\alpha\in A$.  Because $A$ is an antichain, $A^\lesup\setminus\{\alpha\}$ is a permutation class not containing $\alpha$.  Therefore $A^\lesup\setminus\{\alpha\}$ has a finite upper growth rate, which serves as an upper bound for the upper growth rate of $A$ itself.
\end{proof}

\bibliographystyle{acm}
\bibliography{../refs}

\end{document}